\documentclass[12pt]{amsart}
\usepackage{amsmath,amssymb,amsthm}
\usepackage{color}
\usepackage{hyperref}

\newtheorem{lem}{Lemma}[section]

\newtheorem{thm}{Theorem}[section]

\theoremstyle{definition}

\theoremstyle{remark}

\theoremstyle{remark}
\newtheorem{remark}{Remark}[section]
\numberwithin{equation}{section}

\newcommand{\N}{{\mathbb N}}

\newcommand{\R}{{\mathbb R}}

\definecolor{blu}{rgb}{0,0,1}

\title[Sharp non-existence results]{Sharp non-existence results of prescribed $L^2$-norm solutions for some class of Schr\"odinger-Poisson and quasilinear equations}

\author[Louis Jeanjean]{Louis Jeanjean}
\address{Louis Jeanjean
\newline\indent
Laboratoire de Math\'ematiques (UMR 6623)
\newline\indent
Universit\'{e} de Franche-Comt\'{e}
\newline\indent
16, Route de Gray 25030 Besan\c{c}on Cedex, France}
\email{louis.jeanjean@univ-fcomte.fr}

\author[Tingjian Luo]{Tingjian Luo}
\address{Tingjian Luo
\newline\indent
Laboratoire de Math\'ematiques (UMR 6623)
\newline\indent
Universit\'{e} de Franche-Comt\'{e}
\newline\indent
16, Route de Gray 25030 Besan\c{c}on Cedex, France}
\email{tingjianluo@gmail.com}

\begin{document}


\begin{abstract}
In this paper we study the existence of minimizers for
$$  \quad  F(u) = \frac{1}{2}\int_{\R^3} |\nabla u|^2 dx + \frac{1}{4}\int_{\R^3}\int_{\R^3}\frac{\left | u(x) \right |^2\left | u(y) \right |^2}{\left | x-y \right |}dxdy-\frac{1}{p}\int_{\R^3}\left | u \right |^p dx$$
on the constraint
$$S(c) = \{u \in H^1(\R^3) : \int_{\R^3}|u|^2 dx = c \},$$
where  $c>0$ is a given parameter. In the range $p \in [3, \frac{10}{3}]$ we explicit a threshold value of $c>0$ separating existence and non-existence of minimizers. We also derive a non-existence result of critical points of $F(u)$ restricted to $S(c)$ when $c>0$ is sufficiently small. Finally, as a byproduct of our approaches, we extend some results of \cite{CJS} where a constrained minimization problem, associated to a quasilinear equation, is considered.
\end{abstract}
\maketitle

\section{Introduction}
The following stationary nonlinear
Schr\"odinger-Poisson equation
 \begin{equation}\label{eq1.1}
-\Delta u - \lambda u + ( |x|^{-1}\ast |u|^2) u - |u|^{p-2}u = 0 \ \ \mbox{in}\ \ \R^3,
 \end{equation}
where $p \in (2,6)$ and  $\lambda \in \R$ has attracted considerable attention in the recent period. Part of the interest is due to the fact that to a pair $(u(x), \lambda)$ solution of (\ref{eq1.1}) corresponds a standing wave $\phi(x)=e^{-i \lambda t}u(x)$ of the evolution equation
\begin{equation}\label{eq1.2}
i \partial_t \phi + \Delta \phi - ( |x|^{-1}\ast |\phi|^2) \phi + |\phi|^{p-2}\phi =0 \ \ \mbox{in}\ \ \mathbb{R^+}\times \R^3.
\end{equation}
This class of Schr\"odinger type equations with a repulsive nonlocal Coulombic potential is obtained by approximation of the  Hartree-Fock equation describing a quantum mechanical system of many particles, see for instance \cite{BA, L2, LS, MA}. For physical reasons solutions are searched in $H^1(\R^3)$.

A first line of study to (\ref{eq1.1}) is to consider $\lambda \in \R$ as a fixed parameter and then to search for a $u \in H^1(\R^3)$ solving (\ref{eq1.1}). In that direction, mainly by variational methods, the existence, non-existence and multiplicity of solutions have been extensively studied by many authors. See, for example, \cite{AP, AR, DM, HK2, HK, R, R2, WZ, ZZ} and the references therein.

In the present paper, motivated by the fact that physicists are often interested in ``normalized solutions", we look for solutions in $H^1(\R^3)$ having a prescribed $L^2-$norm. More precisely, for given $c>0$ we look to
$$(u_c,\lambda_c)\in H^1(\R^3)\times \R \ \mbox{ solution of }\ \eqref{eq1.1} \mbox{ with } \|u_c\|_{L^2(\R^3)}^2=c.$$
In this case, a solution $u_c\in H^1(\R^3)$ of  \eqref{eq1.1} can be obtained as a constrained critical point of the functional
$$F(u):=\frac{1}{2}\left \| \triangledown u \right \|_{L^{2}(\R^3)}^2+\frac{1}{4}\int_{\R^3}\int_{\R^3}\frac{\left | u(x) \right |^2\left | u(y) \right |^2}{\left | x-y \right |}dxdy-\frac{1}{p}\int_{\R^3}\left | u \right |^pdx $$
\noindent
on the constraint  $$S(c):= \{u \in H^1(\R^3) :\   \left \| u \right \|_{L^{2}(\R^3)}^2=c,\  c>0 \}.$$
The parameter $\lambda_c \in \R$, in this approach,  can't be fixed any longer and it will appear as a Lagrange parameter.

It is well known, see for example  \cite{R}, that for any $p \in (2,6)$, $F(u)$ is a well defined and $C^1$-functional. We set
$$m(c):=\inf_{u\in S(c)} F(u).$$

It is standard that minimizers of $m(c)$ are exactly critical points of $F(u)$ restricted to $S(c)$, and thus solutions of \eqref{eq1.1}. Also it can be checked in many cases that the set of minimizers is orbitally stable under the flow of (\ref{eq1.2}). Thus the search of minimizers can provide us some information on the dynamics of (\ref{eq1.2}). \medskip

By scaling arguments, see Remark \ref{rem1}, it is readily seen that for any $c \in (0, \infty)$, $m(c)\in (-\infty, 0]$ if $p\in (2,\frac{10}{3})$ and $m(c)=-\infty$ if $p\in (\frac{10}{3},6)$. When $m(c)> -\infty$, the existence of minimizers of $m(c)$ has been studied in \cite{BS1}  \cite{BS} \cite{SS}, see also \cite{HK2} for a closely related problem. In \cite{SS}, the authors prove the existence of minimizers  when $p=\frac{8}{3}$ and $c \in (0, c_0)$ for a suitable $c_0>0$. It is shown in \cite{BS} that a minimizer exists if $p\in (2,3)$ and $c>0$ is small enough, and in \cite{BS1} that when $p\in (3,\frac{10}{3})$, $m(c)$ admits a minimizer for any $c>0$ sufficiently large. In addition, when $p \in (\frac{10}{3}, 6)$,  though $m(c)=-\infty$ for all $c>0$, \cite{BJL} shows that there exists, for $c>0$ small enough, a critical point of $F(u)$ constrained on  $S(c)$, at a strictly positive energy level. This  critical point is a least energy solution in the sense that it minimizes $F(u)$ on the set of solutions having this $L^2$-norm. It is proved as well in  \cite{BJL} that it is orbitally unstable.\medskip

The first aim of this paper is to establish non-existence results of minimizers and more generally of constrained critical points of $F(u)$ on $S(c)$ in the range $p\in [3, \frac{10}{3}]$. As we shall see our results are sharp in the sense that we explicit a threshold value of $c>0$ separating existence and non-existence of minimizers.

We first present a detailed study of the function $c \to m(c)$ when $p \in [3, \frac{10}{3}]$. This study is, we believe, interesting for itself, but it is also a key to establish the existence or the non-existence of minimizers.  Let
\begin{eqnarray}\label{1.1}
c_1=\inf\{c>0 \ :\ m(c)<0 \}.
\end{eqnarray}
\begin{thm}\label{lm1.1}
$(\mathbb{I})$ When $p \in (3, \frac{10}{3})$ we have
\begin{itemize}
\item[(i)] $c_1 \in (0, \infty)$;
\item[(ii)]  $m(c)=0$, as $ c \in (0,  c_1]$;
\item[(iii)] $m(c)<0$ and is strictly decreasing about $c$, as $c  \in (c_1, \infty)$.
\end{itemize}
$(\mathbb{II})$ When $p=3$ or $p=\frac{10}{3}$ we have
\begin{itemize}
\item[(iv)] When $p=3$, $m(c)=0$ for all $c>0$;
\item[(v)] When $p=\frac{10}{3}$, we denote
\begin{eqnarray}\label{1.1'}
c_2=\inf\{c>0 \ :\ \exists\ u\in S(c) \mbox{ such that } F(u)\leq 0 \},
\end{eqnarray}
then $c_2 \in (0, \infty)$ and
\begin{eqnarray}\label{1.1''}
\left\{\begin{matrix}
m(c)= 0, &\mbox{as}& c \in (0, c_2);\\
m(c)=-\infty, &\mbox{as}& \ c \in (c_2, \infty).
\end{matrix}\right.
\end{eqnarray}
\end{itemize}
\end{thm}

Our result concerning the existence or non-existence of a minimizer is


\begin{thm}\label{th1.1}
\begin{itemize}
\item[(i)] When $p \in (3, \frac{10}{3})$, $m(c)$ has a minimizer if and only if $c \in [c_1, \infty)$.
\item[(ii)] When $p=3$ or $p=\frac{10}{3}$, $m(c)$ has no minimizer for any $c>0$.
\end{itemize}
\end{thm}

\begin{remark}\label{rem1}
One always has $m(c) \leq 0$ for any $c>0$. Indeed let $u \in S(c)$ be arbitrary and consider the scaling  $u^t(x)=t^{\frac{3}{2}}u(tx)$. We have $u^t \in S(c)$ for any $t >0$ and also
$$F(u^t) = \frac{t^2}{2}\int_{\R^3} |\nabla u|^2 dx + \frac{t}{4} \int_{\R^3}\int_{\R^3}\frac{\left | u(x) \right |^2\left | u(y) \right |^2}{\left | x-y \right |}dxdy  - \frac{t^{\frac{3}{2}(p-2)}}{p} \int_{\R^3}|u|^p dx.$$
Thus $F(u^t) \to 0$ as $t \to 0$ and the conclusion follows.
\end{remark}

\begin{remark}\label{rem2}
In \cite{HK2, HK} the minimization problem on $S(c)$ for the functional
$$  \quad  F_{a,b}(u) : = \frac{1}{2}\int_{\R^3} |\nabla u|^2 dx + \frac{a}{4}\int_{\R^3}\int_{\R^3}\frac{\left | u(x) \right |^2\left | u(y) \right |^2}{\left | x-y \right |}dxdy-\frac{b}{p}\int_{\R^3}\left | u \right |^p dx$$
is considered. When $p=3$ it is proved that for each $a>0$, there exists a $b_0>0$ such that if $b> b_0$ then a minimizer exists for all $c>0$ (see Theorem 1.4 of \cite{HK2}). Theorem \ref{th1.1} (ii) implies that when $a=1$, necessarily $b_0>1$.
\end{remark}

\begin{remark}\label{rem3}
Theorem \ref{th1.1} provides a complete answer to the issue of minimizers for $F(u)$ on $S(c)$ when $ p \in [3, \frac{10}{3}]$. When $ p \in (2,3)$, this question is still open. In \cite{BS} it is proved that a minimizer exists when $c>0$ is sufficiently small. However even if $m(c) <0$, for any $ c >0$ and any minimizing sequence is bounded, we still do not know what happen for an arbitrary value of $c>0$. In trying to develop a minimization process one faces the difficulty to remove the possible dichotomy of the minimizing sequences. Also when $p \in (\frac{10}{3}, 6)$ the existence of a least energy solution is only established for $c>0$ small (see \cite{BJL}). In \cite{BJL} however and even if the result is still to be proved, strong indications are given that there do not exist least energy critical points of $F(u)$ constrained to $S(c)$ when $c>0$ is large.
\end{remark}

In addition to the non-existence results of Theorem \ref{th1.1} we also show that, taking eventually $c>0$ smaller, there are no critical points of $F(u)$ on $S(c)$. Precisely

\begin{thm}\label{th1.2}
When $ p \in (3, \frac{10}{3}]$, there exists $\bar{c}>0$ such that for any $c \in (0, \bar{c})$, there are no critical points of $F(u)$ restricted to $S(c)$. When $p=3$, for all $c>0$, $F(u)$ does not admit critical points on the constraint $S(c)$.
\end{thm}

\begin{remark}\label{rem4}
Theorem \ref{th1.2} is, up to our knowledge, the only result where a non-existence result of  small  $L^2$ norm solutions is established for (\ref{eq1.1}). Note however that in \cite{HK1, R} it was independently proved that when $p \in (2, 3]$ there exists a $\lambda_0 <0$ such that (\ref{eq1.1}) has only trivial solution when $\lambda \in (- \infty, \lambda_0).$
\end{remark}

Another aim of this paper is to clarify and extend some results contained in \cite{CJS} where a constrained minimization problem associated to a quasilinear equation is considered. Actually in \cite{CJS} one looks for minimizers of
\begin{eqnarray}\label{4.1}
\overline{m}(c)= \inf_{\sigma(c)}\mathcal{E}(u),
\end{eqnarray}
where
\begin{eqnarray}\label{4.2}
\mathcal{E}(u)= \frac{1}{2} \int_{\R^N}|\nabla u|^2 dx + \int_{\R^N}|u|^2|\nabla u|^2dx- \frac{1}{p+1}\int_{\R^N}|u|^{p+1}dx,
\end{eqnarray}
and $$\sigma(c)= \{u\in H^1(\R^N)\ :\ \int_{\R^N}|u|^2|\nabla u|^2dx <\infty \mbox{ with } \|u\|_{L^2(\R^N)}^2=c \}.$$
Here $N \in \N^+$ and we focus on the range $p \in [1+\frac{4}{N}, 3+\frac{4}{N}]$.  Let
$$c(p,N) = \inf \{ c>0 : \overline{m}(c) <0\}.$$
\begin{thm}\label{th4.3}
\begin{itemize}
\item[(i)] If $ p \in [1+\frac{4}{N}, 3+\frac{4}{N})$, we have
\begin{itemize}
\item[a)]
 $c(p,N) \in (0, \infty);$
 \item[b)]
  $\overline{m}(c)=0$ if $c \in (0, c(p,N)]$;
\item[c)] $\overline{m}(c)<0$ if $c \in (c(p,N), \infty)$ and is strictly decreasing about $c$, as $c \in (c(p,N), \infty).$
\end{itemize}
\item[(ii)] If $p \in [1 + \frac{4}{N} , 3+\frac{4}{N})$, the mapping $c\longmapsto \overline{m}(c)$ is continuous at each $c>0$.
\item[(iii)] If $p=3+\frac{4}{N}$, we denote
\begin{eqnarray}\label{4.3'}
c_N=\inf\{c>0 \ :\ \exists\ u\in \sigma(c) \mbox{ such that } \mathcal{E}(u) \leq 0 \},
\end{eqnarray}
then $c_N \in (0,\infty)$ and
\begin{eqnarray}\label{4.3''}
\left\{\begin{matrix}
\overline{m}(c)= 0, &\mbox{as}& c \in (0, c_N);\\
\overline{m}(c)=-\infty, &\mbox{as}& \ c \in (c_N, \infty).
\end{matrix}\right.
\end{eqnarray}
\end{itemize}
\end{thm}

Concerning the existence or non-existence of minimizers we have
\begin{thm}\label{th4.33}
\begin{itemize}
\item[(i)] If $ p \in (1+\frac{4}{N}, 3+\frac{4}{N})$, then $\overline{m}(c)$ admits a minimizer if and only if $c \in [c(p,N), \infty).$
\item[(ii)] If $p = 3 + \frac{4}{N}$, $\overline{m}(c)$ has no minimizer for all $c \in (0, \infty).$
\end{itemize}
\end{thm}

\begin{remark}\label{rem-critical}
We note that in \cite{CJS} it was proved that when $p\in (1,1+\frac{4}{N})$, for all $c>0$, $\overline{m}(c)<0$ and $\overline{m}(c)$ admits a minimizer. When $p=1+\frac{4}{N}$, we conjecture that the conclusion of Theorem \ref{th4.33} (i) also holds. As for $p\in (3+\frac{4}{N}, \infty)$, $\overline{m}(c)=-\infty$ for any $c>0$.
\end{remark}

\begin{remark}\label{rem4}
We point out that parts of Theorems \ref{th4.3} and \ref{th4.33} are already contained in Theorem 1.12 of \cite{CJS}. However, on one hand we provide here additional information. In particular we settle the question of existence for the threshold value $c(p,N)$ which requires a special treatment. On the other hand some statements of Theorem 1.12  are wrong, in particular concerning the case $p = 3 + \frac{4}{N}$. There are also some gaps in the proofs of \cite{CJS}. In particular it is not proved completely that there are no minimizer when $c \in (0, c(p,N))$.
\end{remark}
\begin{remark}\label{rem5}
In \cite{CS}, the minimization problem (\ref{4.1}) is studied and the question of finding explicit bounds on $c(p,N)$ and $c_N$ is addressed by a combination of analytical and numerical arguments in dimension $N=3$. In particular, when $p=3+\frac{4}{N}$ a $c_b>0$ such that $\overline{m}(c)=0$ if $0<c\leq c_b$ and a $c^b>0$ such that $\overline{m}(c)=-\infty$ if $c> c^b$ are explicitly given (see Proposition 2.1, points (4) and (5) of \cite{CS}). Their values are $c_b\approx 19.73$ and $c^b \approx 85.09$. Theorem \ref{th4.3} (iii) complements these results  showing that the change from $\overline{m}(c)=0$ to $\overline{m}(c)=-\infty$ occurs abruptly at the value $c_N$. We also point out that our results hold for any dimension $N \in \N^+$.
\end{remark}

Finally, similarly to Theorem \ref{th1.2} we obtain
\begin{thm}\label{th4.333}
Assume that $ p \in [1+\frac{4}{N}, 3+\frac{4}{N}]$ holds, then there exists a $\hat{c}>0$ such that for all $c\in (0, \hat{c})$, the functional $\mathcal{E}(u)$, restricted to $\sigma(c)$, has no critical points.
\end{thm}


\noindent \textbf{Acknowledgement:} The authors  thank the referee for its comments which have permitted to simplify several proofs in the paper. \\

\noindent \textbf{Notations:} For convenience we set
$$A(u):=\int_{\R^3} |\nabla u|^{2}dx,\ \ \ B(u):=\int_{\R^3}\int_{\R^3}\frac{\left | u(x) \right |^2\left | u(y) \right |^2}{\left | x-y \right |}dxdy$$
$$ \ C(u):=\int_{\R^3} |u|^{p}dx,\ \ \ D(u):= \int_{\R^3} | u|^{2}dx.$$
Then
\begin{eqnarray}\label{func}
F(u)=\frac{1}{2}A(u)+\frac{1}{4}B(u)-\frac{1}{p}C(u).
\end{eqnarray}
Also we denote by $||\cdot||_p$ the standard norm on $L^p(\R^N)$. Throughout the paper we shall denote by $C>0$ various positive constants which may vary from one line to another and which are not important for the analysis of the problem.

\section{Preliminary results}\label{Sec1}
To obtain our non-existence results we use the fact that any critical point of $F(u)$ on $S(c)$ satisfies $Q(u)=0$ where
\begin{eqnarray*}
 Q(u):=\int_{\R^3} |\nabla u|^{2}dx  +\frac{1}{4}\int_{\R^3}\int_{\R^3}\frac{\left | u(x) \right |^2\left | u(y) \right |^2}{\left | x-y \right |}dxdy -\frac{3(p-2)}{2p}\int_{\R^3} |u|^{p}dx.
\end{eqnarray*}
Indeed we have

\begin{lem}\label{lm2.3}
If $u_0$ is a critical point of $F(u)$ on $S(c)$, then $Q(u_0)=0$.
\end{lem}

\begin{proof} First we denote
\begin{eqnarray}\label{2.6}
I_{\lambda}(u)&:=&\langle S_{\lambda }'(u),u \rangle =A(u)-\lambda D(u)+B(u)-C(u), \\
P_{\lambda}(u)&:=&\frac{1}{2}A(u)-\frac{3}{2}\lambda D(u)+\frac{5}{4}B(u)-\frac{3}{p}C(u).
\end{eqnarray}
Here $\lambda \in \R$ is a parameter and $S_{\lambda }(u)$ is the energy functional corresponding to the equation (\ref{eq1.1}), i.e.
\begin{eqnarray}\label{2.8}
S_{\lambda }(u): =\frac{1}{2}A(u)-\frac{\lambda}{2}D(u)+\frac{1}{4}B(u)-\frac{1}{p}C(u).
\end{eqnarray}
Clearly $S_{\lambda}(u)=F(u)-\frac{\lambda}{2}D(u)$ and simple calculations imply that
\begin{eqnarray}\label{2.9}
\frac{3}{2}I_{\lambda}(u)-P_{\lambda}(u)=Q(u).
\end{eqnarray}
Now from \cite{DM} or Theorem 2.2 of \cite{R}, we know that $P_{\lambda}(u)=0$ is a Pohozaev identity for the Schr\"odinger-Poisson equation (\ref{eq1.1}). In particular any critical point $u$ of $S_{\lambda}(u)$ satisfies $P_{\lambda}(u)=0$. \medskip

On the other hand, since $u_0$ is a critical point of $F(u)$ restricted to $S(c)$, there exists a Lagrange multiplier $\lambda_0 \in \R$, such that
$$F'(u_0) = \lambda_0  u_0.$$
Thus for any $\phi \in H^1(\R^3)$,
\begin{eqnarray}\label{2.10}
\langle S_{\lambda_0}'(u_0), \phi \rangle = \langle F'(u_0)- \lambda_0 u_0, \phi \rangle =0,
\end{eqnarray}
which shows that $u_0$ is also a critical point of $S_{\lambda_0}(u)$. Hence
$$P_{\lambda_0}(u_0)=0,  \quad I_{\lambda_0}(u_0) =\langle S_{\lambda_0}'(u_0), u_0 \rangle = 0,$$
and $Q(u_0)=0$ follows from (\ref{2.9}).
\end{proof}

We now give an estimate on the nonlocal term, which is useful to control the functionals $F(u)$ and $Q(u)$.
\begin{lem}\label{lm2.1}
When $ p \in [3,4]$, there exists a constant $C>0$, depending only on $p$, such that, for any $u \in S(c)$,
\begin{eqnarray}\label{2.1}
\int_{\R^3}\int_{\R^3}\frac{\left | u(x) \right |^2\left | u(y) \right |^2}{\left | x-y \right |}dxdy \geq -\frac{1}{16\pi}\left \| \triangledown u \right \|_{2}^2+C\frac{ \left \| u \right \|_{p}^{\frac{p}{4-p}}}{\left \| \triangledown u \right \|_{2}^{\frac{3(p-3)}{4-p}}\left \| u \right \|_{2}^{\frac{p-3}{4-p}}}.
\end{eqnarray}
\end{lem}

\begin{proof} Since $p \in [3,4]$, by interpolation, we have
\begin{eqnarray}\label{2.2}
\left \| u \right \|_{p}^p \leq \left \| u \right \|_{3}^{3(4-p)}\left \| u \right \|_{4}^{4(p-3)}.
\end{eqnarray}
In addition, since $(|x|^{-1}\ast |u|^2)\in \mathcal{D}^{1,2}(\R^3)$ solves the equation
\begin{eqnarray}\label{2.2.1}
-\Delta \Phi=4 \pi |u|^2\ \ \ \mbox{in } \R^{3},
\end{eqnarray}
on one hand multiplying (\ref{2.2.1}) by $(|x|^{-1}\ast |u|^2)\in \mathcal{D}^{1,2}(\R^3)$ and integrating we get
\begin{eqnarray}\label{2.2.2}
4\pi \int_{\R^3}(|x|^{-1}\ast |u|^2)|u|^2 dx = \int_{\R^3}|\nabla (|x|^{-1}\ast |u|^2)|^2 dx.
\end{eqnarray}
On the other hand, multiplying (\ref{2.2.1}) by $|u|$ and integrating we get for any $\eta >0$,
\begin{eqnarray}\label{kikuchi}
4\pi \eta \int_{\R^3}|u|^3 dx &=& \eta \int_{\R^3}- \Delta (|x|^{-1}\ast |u|^2)|u|dx  \nonumber \\
&\leq& \eta \int_{\R^3}\nabla (|x|^{-1}\ast |u|^2)\cdot \nabla |u|dx \\
&\leq & \int_{\R^3}|\nabla (|x|^{-1}\ast |u|^2)|^2 dx + \frac{\eta ^2}{4}\int_{\R^3}|\nabla u|^2dx. \nonumber
\end{eqnarray}
Thus, taking $\eta=1$ in (\ref{kikuchi}) it follows from (\ref{2.2.2}) and (\ref{kikuchi}) that
\begin{eqnarray}\label{2.3}
\int_{\R^3}\left | u \right |^3dx \leq \int_{\R^3}\int_{\R^3}\frac{\left | u(x) \right |^2\left | u(y) \right |^2}{\left | x-y \right |}dxdy + \frac{1}{16\pi} \left \| \triangledown u \right \|_{2}^2.
\end{eqnarray}
Now, using Gagliardo-Nirenberg's inequality, there exists a constant $C>0$, depending only on $p$, such that
\begin{eqnarray}\label{2.4}
\int_{\R^3}\left | u \right |^4dx \leq C\left \| \triangledown u \right \|_{2}^3\left \| u \right \|_2.
\end{eqnarray}
Taking (\ref{2.3}) and (\ref{2.4}) into (\ref{2.2}), we obtain
\begin{equation*}
\left \| u \right \|_{p}^p \leq C\left ( \int_{\R^3} \int_{\R^3} \frac{\left | u(x) \right |^2 \left | u(y) \right |^2}{\left | x-y \right |}  dxdy+\frac{1}{16\pi}\left \| \nabla u \right \|_{2}^2 \right )^{(4-p)}\left \| \nabla u \right \|_{2}^{3(p-3)}\left \| u \right \|_{2}^{(p-3)},
\end{equation*}
which implies (\ref{2.1}).
\end{proof}

The estimate (\ref{2.1}) leads to a lower bound on $Q(u)$.
\begin{lem}\label{lm_q} When $p \in (3,  \frac{10}{3})$,  there exists a constant $C>0$, depending only on $p$, such that, for any $u \in S(c)$
\begin{eqnarray}\label{lm_q.1}
Q(u)\geq\frac{64\pi-1}{64\pi}A(u)- C\cdot A(u)^{\frac{3}{2}} \cdot c^{\frac{1}{2}}.
\end{eqnarray}
\end{lem}

\begin{proof} By Lemma \ref{lm2.1} there exists a constant $C>0$ depending only on $p$, such that, for any $u \in S(c)$,
\begin{equation}\label{raaa}
Q(u)\geq\frac{64\pi-1}{64\pi}A(u)+C\cdot \frac{ C(u)^{\frac{1}{4-p}} }{A(u)^{\frac{3(p-3)}{2(4-p)}}\cdot D(u)^{\frac{p-3}{2(4-p)}}} - \frac{3(p-2)}{2p}C(u).
\end{equation}
To obtain (\ref{lm_q.1}) from (\ref{raaa}) we introduce the auxiliary function
$$f_K(x)=\left (\frac{64\pi-1}{64\pi} \right )K+ D\cdot x^{\frac{1}{4-p}}-\frac{3(p-2)}{2p}\cdot x, \quad x>0$$
with $D=C\cdot\left (  K^{\frac{3(p-3)}{2(4-p)}}\cdot c^{\frac{p-3}{2(4-p)}} \right )^{-1}$. Its study will provide us an estimate independent of $C(u)$. Clearly
\begin{eqnarray*}
f_K'(x)&=&D\cdot \frac{1}{4-p}\cdot x^{\frac{p-3}{4-p}}-\frac{3(p-2)}{2p}, \\
f_K''(x)&=&D\cdot \frac{1}{4-p}\cdot \frac{p-3}{4-p}\cdot x^{\frac{p-3}{4-p}-1}>0, \quad \mbox{ for all }  x>0.
\end{eqnarray*}
Therefore $f_K(x)$ has the unique global minimum at
$$\bar{x}=\left ( \frac{3(p-2)(4-p)}{2p D} \right )^{\frac{4-p}{p-3}},$$
and
\begin{eqnarray*}
f_K(\bar{x})&=& \frac{64\pi-1}{64\pi}K+ D\cdot \left ( \frac{3(p-2)(4-p)}{2p D} \right )^{\frac{1}{p-3}}-\frac{3(p-2)}{2p}\cdot \left ( \frac{3(p-2)(4-p)}{2p D} \right )^{\frac{4-p}{p-3}} \\
&=& \frac{64\pi-1}{64\pi}K-\left ( \frac{3(p-2)(4-p)}{2p}  \right )^{\frac{1}{p-3}}\cdot   \frac{p-3}{4-p}\cdot  D^{\frac{p-4}{p-3}}\\
&=&\frac{64\pi-1}{64\pi}K-\left ( \frac{3(p-2)(4-p)}{2p}  \right )^{\frac{1}{p-3}}\cdot   \frac{p-3}{4-p}\cdot C^{\frac{p-4}{p-3}}\cdot  K^{\frac{3}{2}}\cdot c^{\frac{1}{2}}.
\end{eqnarray*}
Thus $f_K(x) \geq f_K(\bar{x})$ for all $x >0$. This, together with (\ref{raaa}) implies \eqref{lm_q.1}.
\end{proof}

Finally we recall the following results obtained in \cite{BS1,BS}.

\begin{lem}\label{lm2.2}
Let $ p \in (3, \frac{10}{3})$, then
\begin{itemize}
\item[(i)] For any $c>0$ such that $m(c) <0$, $m(c)$ admits a minimizer.
\item[(ii)] There exists $d>0$, such that for all $c\in (d,\infty)$, $m(c)<0$.
\item[(iii)] The function $c \mapsto m(c)$ is continuous at each $c>0$.
\end{itemize}
\end{lem}

\begin{remark}
Points (i) and (ii) of Lemma \ref{lm2.2} are proved in \cite{BS1}. Concerning Point (iii), in  \cite{BS} the authors prove the continuity of $m(c)$ about $c>0$ when $p \in (2,3)$. However inspecting their proof reveals that it also holds for $ p \in [3, \frac{10}{3})$.
\end{remark}

\section{Proofs of the main results}

We first give the following non-existence result.
\begin{lem}\label{lm3.1}
When $p \in (3,\frac{10}{3})$, there exists a $c_3>0$, such that $m(c)$ has no minimizer for all $c\in (0,c_3)$.
\end{lem}
\begin{proof} Let us assume by contradiction that there exist sequences $\{c_n\} \subset \R^+$, with $c_n \to 0$ as $n \to \infty$, and $\{u_n\} \subset S(c_n)$ such that $F(u_n)=m(c_n)$. Then by Lemma \ref{lm2.3}, $Q(u_n)=0$ for any $n \in \mathbb{N}^+$. \smallskip

Since $m(c)\leq 0$ for any $c>0$, see Remark \ref{rem1}, we know that $F(u_n)\leq 0$. Thus
\begin{eqnarray}\label{3.1}
\frac{1}{2}A(u_n)+\frac{1}{4}B(u_n) &\leq& \frac{1}{p}C(u_n)   \nonumber \\
&\leq& \frac{C}{p}A(u_n)^{\frac{3}{4}(p-2)}\cdot D(u_n)^{\frac{6-p}{4}},
\end{eqnarray}
by Gagliardo-Nirenberg's inequality. Since $ p \in (3, \frac{10}{3})$, $1> \frac{3}{4}(p-2)$ and thus (\ref{3.1}) implies that
\begin{eqnarray}\label{3.2}
A(u_n)\rightarrow 0, \ \mbox{as} \ n \to \infty.
\end{eqnarray}
Now due to (\ref{3.2}) and  Lemma \ref{lm_q}, when $n \in \mathbb{N}^+$ is sufficiently large,
\begin{eqnarray*}
Q(u_n)&\geq& \frac{64\pi-1}{64\pi}A(u_n)-C \cdot A(u_n)^{\frac{3}{2}}\cdot c_n^{\frac{1}{2}} \\
&\geq& \frac{64\pi-1}{64\pi}A(u_n)-C \cdot A(u_n)^{\frac{3}{2}}>0.
\end{eqnarray*}
Obviously this contradicts Lemma \ref{lm2.3} and this ends the proof.
\end{proof}

The following lemma is crucial to establish a precise threshold between existence and non-existence.
\begin{lem}\label{lm3.2}
Assume that $p \in (3, \frac{10}{3})$ holds. For any  $c>0$ such that $m(c)<0$ or such that $m(c)=0$ and $m(c)$ has a minimizer we have
\begin{eqnarray*}
m(t c)< t m(c), \ \mbox{for all} \ t>1.
\end{eqnarray*}
\end{lem}

\begin{proof}
By Lemma \ref{lm2.2} (i) without restriction we can assume that $m(c) \leq 0$ admit a minimizer $u_c\in S(c)$. We set $(u_c)_t(x)=t^2u_c(tx)$ for $t>1$. Then $D((u_c)_t)=tD(u_c)=tc$, and since $2p-6>0$ in case of $p\in (3, 10/3]$ and $C(u_c)>0$, we obtain
\begin{eqnarray}\label{lm3.2_1}
m(t c)\leq F((u_c)_t)&=& t^3\cdot \left( \frac{1}{2}A(u_c)+\frac{1}{4}B(u_c)-\frac{t^{2p-6}}{p}C(u_c)  \right ) \nonumber \\
&<&t^3\cdot \left( \frac{1}{2}A(u_c)+\frac{1}{4}B(u_c)-\frac{1}{p}C(u_c)  \right )  \\
&=&t^3\cdot F(u_c)=t^3 m(c).\nonumber
\end{eqnarray}
Since $m(c) \leq 0$ and $t>1$, we conclude from (\ref{lm3.2_1}) that $m(t c)< t^3 m(c) \leq t m(c)$.
\end{proof}

In the case $p=\frac{10}{3}$ we first have
\begin{lem}\label{lm3.5}
When $p=\frac{10}{3}$, we have $c_2 \in (0, \infty)$, where $c_2$ is given by (\ref{1.1'}).
\end{lem}
\begin{proof}
First observe that by Gagliardo-Nirenberg's inequality, when $p=\frac{10}{3}$ we have
\begin{eqnarray}\label{lm3.5.1}
C(u)\leq C\cdot A(u)\cdot c^{\frac{2}{3}},\quad \mbox{ for all } u\in S(c),
\end{eqnarray}
where $ C>0$ independent of $c>0$. Thus for any $u\in S(c)$, there holds
\begin{eqnarray}\label{lm3.5.2}
F(u)&\geq& \frac{1}{2}A(u)+\frac{1}{4}B(u)-\frac{3}{10} C\cdot A(u)\cdot c^{\frac{2}{3}} \nonumber \\
&\geq& A(u)\left ( \frac{1}{2}- \frac{3}{10} C \cdot c^{\frac{2}{3}} \right ).
\end{eqnarray}
Thus $F(u)>0$, for all $ u\in S(c)$ if $c>0$ is sufficiently small and it proves that $c_2>0$.

Now take $u_1\in S(1)$ arbitrary and consider the scaling
\begin{eqnarray}\label{lm3.5.3}
u_t(x)=t^2 u_1(tx), \quad  \mbox{ for all } t>0.
\end{eqnarray}
Then $u_t\in S(t)$ and
\begin{eqnarray}\label{lm3.5.4}
F(u_t)&=&\frac{t^3}{2}A(u_1)+\frac{t^3}{4}B(u_1)-\frac{3}{10}t^{\frac{11}{3}}C(u_1) \nonumber \\
&=&t^3 \left(\frac{1}{2}A(u_1)+\frac{1}{4}B(u_1)-\frac{3}{10}t^{\frac{2}{3}}C(u_1) \right).
\end{eqnarray}
This shows that $F(u_t)<0$ for $t>0$ large enough and proves that $c_2<\infty$.
\end{proof}

We can now give the

\begin{proof}[Proof of Theorem \ref{lm1.1}]
First we prove that $c_1>0$ by contradiction. If we assume that $c_1=0$ then, from the definition of $c_1$, $m(c)<0$ for all $c>0$. Thus Lemma \ref{lm2.2} (i) implies the existence of a minimizer for any $c>0$ and this contradicts Lemma \ref{lm3.1}. Additionally Lemma \ref{lm2.2} (ii) shows that $c_1 <\infty $, thus Point (i) follows. To prove Point (ii) we observe that since $m(c) \leq 0$ for all $c>0$, from the definition of $c_1 >0$ it follows that $m(c) = 0$ if $c \in (0, c_1)$. Using the continuity of $c \mapsto m(c)$, see Lemma \ref{lm2.2} (iii), we obtain that $m(c_1)=0$ and then Point (ii) holds. Point (iii) is a direct consequence of Lemma \ref{lm3.2} and of the definition of $c_1>0$.

Concerning Point (iv), it is enough to show that if $p=3$, for any $c>0$ one has
\begin{eqnarray}\label{3.11.0}
F(u)>0,\quad \mbox{ for all } u\in S(c).
\end{eqnarray}
Indeed, since $m(c)\leq 0$ for all $c>0$,  (\ref{3.11.0}) implies immediately Point (iv). To check (\ref{3.11.0}), we use (\ref{kikuchi}) with $\eta = 4/3$.  From (\ref{2.2.2}) and (\ref{kikuchi}) we then get
$$\frac{1}{4}\int_{\R^3}\int_{\R^3}\frac{|u(x)|^2|u(y)|^2}{|x-y|}dx dy \geq - \frac{1}{36 \pi} ||\nabla u||_2^2 + \frac{1}{3}||u||^3_3.$$
Thus when $p=3$, for any $u\in S(c)$,
\begin{eqnarray*}
F(u)\geq \frac{1}{2}||\nabla u||_2^2-\frac{1}{36 \pi} ||\nabla u||_2^2 >0
\end{eqnarray*}
and (\ref{3.11.0}) holds.

Finally since, by Lemma \ref{lm3.5}, $c_2 \in (0,\infty)$, to prove Point (v) it is enough to verify (\ref{1.1''}). From the definition of $c_2$, it follows directly that $m(c)=0$ for any $c \in (0, c_2)$. Now if $c \in (c_2, \infty)$, we first claim that there exists a $v\in S(c)$ such that $F(v)\leq 0$. Indeed if we assume that
$F(u)>0$ for all $ u\in S(c)$ we reach a contradiction as follows. For an arbitrary $\hat{c}\in [c_2, c ) $ taking any $u\in S(\hat{c})$ we scale it as in (\ref{lm3.5.3}) where $t=c/ \hat{c}$. Then $u_t\in S(c)$ and it follows from (\ref{lm3.5.4}) that
$F(u_t)\leq t^3F(u).$ This implies that $F(u)>0$ for all $u\in S(\hat{c})$ and since $\hat{c}\in [c_2, c ) $ is arbitrary this contradicts the definition of $c_2 > 0$.
Hence, for any $c \in (c_2, \infty)$, there exists a $u_0\in S(c)$ such that $F(u_0)\leq 0$.

Consider now the scaling
\begin{eqnarray}\label{3.11.1}
u^{\theta}(x)= \theta ^{\frac{3}{2}}u_0(\theta x), \quad \mbox{ for all }  \theta >0.
\end{eqnarray}
We have $u^{\theta} \in S(c)$ for all $\theta >0$ and
\begin{eqnarray}\label{3.11.2}
F(u^{\theta}) &=& \frac{\theta ^2}{2}A(u_0)+ \frac{\theta}{4}B(u_0)-\frac{10}{3}\theta ^2 C(u_0) \nonumber \\
&=& \frac{\theta}{4}B(u_0)- \left(\frac{10}{3}C(u_0)- \frac{1}{2}A(u_0) \right)\cdot \theta ^2.
\end{eqnarray}
Since $F(u_0)\leq 0$, necessarily
$$\frac{10}{3} C(u_0)- \frac{1}{2}A(u_0)>0.$$
Thus we see from (\ref{3.11.2}) that $\lim_{\theta \to \infty}F(u^{\theta})=-\infty$ and $m(c)=-\infty$ follows. At this point the proof of the theorem is completed.
\end{proof}

Before giving the proof of Theorem \ref{th1.1} we consider the case where $c= c_1$ that requires a special treatment.

\begin{lem}\label{lm5.1}
Assume that $p \in (3, \frac{10}{3})$ holds. Then $m(c_1)$ admits a minimizer.
\end{lem}

\begin{proof} Let $k_n:=c_1+ 1/n,$ for all  $n\in \N^+$. We have $k_n \to c_1$ and thus, by Lemma 2.4 (iii), $m(k_n) \to m(c_1)=0$. Furthermore, by Theorem \ref{lm1.1} (iii) and Lemma
\ref{lm2.2} (i) we know that for each $n\in \N^+$, $m(k_n)<0$ and $m(k_n)$ admits a minimizer $u_n$. Now we claim that the sequence $\{u_n\}$ is bounded in $H^1(\R^3)$. Indeed, by Gagliardo-Nirenberg's inequality, we have
\begin{eqnarray*}
\frac{1}{2}A(u_n)+\frac{1}{4}B(u_n)&=&\frac{1}{p}C(u_n)+F(u_n) \\
&\leq& CA(u_n)^{\frac{3(p-2)}{4}} k_n^{\frac{6-p}{4}} + m(k_n).
\end{eqnarray*}
This implies that $\{A(u_n)\}$ is bounded, since $m(k_n) \leq 0$ and  $1>3(p-2)/4$. Thus we conclude that $\{u_n\}$ is bounded in $H^1(\R^3)$.

Now we claim that $C(u_n) \nrightarrow 0$. By contradiction let us assume that $C(u_n)\to 0$ as $n\to \infty$. Since $F(u_n)\to m(c_1)=0$ it then follows that
\begin{eqnarray}\label{5.1.1}
A(u_n)\to 0\ \mbox{ and }\ B(u_n)\to 0,\ \mbox{as }\ n\to \infty.
\end{eqnarray}
Now, similarly to the proof of Lemma \ref{lm_q}, using \eqref{2.1}, we can estimate  $F(u)$ from below by
\begin{eqnarray}\label{5.1.2}
F(u)\geq\frac{32\pi-1}{64\pi}A(u)-C\cdot A(u)^{\frac{3}{2}} \cdot c^{\frac{1}{2}}, \quad \mbox{ for all } u \in S(c)
\end{eqnarray}
where $C>0$ is constant, depending only on $p$. In particular
\begin{equation}\label{rajout3}
F(u_n)\geq  A(u_n) \left( \frac{32\pi-1}{64\pi}- C\cdot A(u_n)^{\frac{1}{2}} \cdot k_n^{\frac{1}{2}}   \right).
\end{equation}
Taking (\ref{5.1.1}) into account, (\ref{rajout3}) implies that $F(u_n) \geq 0$ for $n \in \N^+$ sufficiently large. This contradicts the fact that $F(u_n)=m(k_n)<0$ for all $n\in \mathbb{N}^+$ and proves the claim.

Now, by Lemma I.1 of \cite{L}, we deduce that $\{u_n\}$ does not vanish. Namely that there exists a constant $\delta >0$ and a  sequence $\{x_n\}\subset \mathbb{R}^3$ such that
$$\int_{B(x_n,1)}|u_n|^2 dx \geq \delta>0,$$
or equivalently
\begin{equation}\label{5.1.3}
\int_{B(0,1)}|u_n(\cdot + x_n)|^2 dx \geq \delta>0.
\end{equation}
Here $B(0,1)$ denotes the ball centered in 0 with radius $r=1$. Now let $v_n(\cdot)=u_n(\cdot + x_n)$. Clearly $\{v_n\}$ is bounded in $H^1(\R^3)$ and thus there exists $v_0\in H^1(\R^3)$ such that
$$ v_n \rightharpoonup  v_0 \ \mbox{ weakly in } H^1(\R^3) \quad \mbox{ and } \quad v_n \rightarrow  v_0 \ \ \mbox{ in } L_{loc}^2(\R^3).$$
We note that $v_0 \neq 0$, since by \eqref{5.1.3}
$$0<\delta \leq \lim_{n\to \infty}\int_{B(0,1)}|v_n|^2dx = \int_{B(0,1)}|v_0|^2dx.$$
Let us prove that $v_0$ is a minimizer of $m(c_1)$. First we show that $F(v_0)=0$. Clearly
\begin{equation}\label{rajout4}
\lim_{n\to \infty}\|v_n\|_2^2=\|v_0\|_2^2+ \lim_{n\to \infty}\|v_n-v_0\|_2^2=c_1
\end{equation}
and using  Lemma \ref{lm2.2} (iii) we deduce from (\ref{rajout4}) that
\begin{equation}\label{rajout5}
\lim_{n \to \infty} F(v_n - v_0) \geq \lim_{n \to \infty} m(||v_n - v_0||_2^2) = m(c_1 - ||v_0||_2^2) = 0.
\end{equation}
Here we make the convention that $m(0)=0$. Now using Lemma 2.2 of \cite{ZZ}, we have
\begin{equation}\label{rajout6}
0 = m(c_1)= \lim_{n\to \infty}F(v_n)= F(v_0)+\lim_{n\to \infty}F(v_n-v_0).
\end{equation}
Since $||v_0||_2^2 \leq c_1$ we have $m(||v_0||_2^2)=0$ and it shows that $F(v_0) <0$ is impossible.  From (\ref{rajout5}) and (\ref{rajout6}) we deduce that $F(v_0) =0$ and that $v_0$ is a minimizer associated to $m(||v_0||_2^2)$. If we assume that $||v_0||_2^2 < c_1$ we get a contradiction with Lemma \ref{lm3.2} since $m(c_1)=0$. Thus necessarily $\|v_0\|_2^2 = c_1$ and this ends the proof.
\end{proof}

\begin{proof}[Proof of Theorem \ref{th1.1}]
To prove Point (i) we assume by contradiction that there exists $\widetilde{c} \in (0, c_1)$ such that $m(\widetilde{c})$ admits a minimizer. Then from the definition of $c_1 >0$ we get that $m(\widetilde{c}) =0$ and Lemma \ref{lm3.2} implies that $m(c) <0$ for any $c > \widetilde{c}$. This contradicts the definition of $c_1 >0$. Now when $c>c_1$ the result clearly follows from Theorem \ref{lm1.1} (iii) and Lemma \ref{lm2.2} (i). Finally the case $c=c_1$ is considered in Lemma \ref{lm5.1}. For Point (ii), first observe that, because of (\ref{3.11.0}), when $p=3$, for any $c>0$, $m(c)$ does not have a minimizer. Then we note that, from the definition of $Q(u)$, it holds, for any $u \in S(c)$,
\begin{eqnarray}\label{3.12.1}
F(u)-\frac{2}{3(p-2)}Q(u)=\frac{3p-10}{6(p-2)}A(u)+\frac{3p-8}{12(p-2)}B(u).
\end{eqnarray}
Taking $p=\frac{10}{3}$  in (\ref{3.12.1}) we obtain
\begin{eqnarray}\label{3.12.2}
F(u)-\frac{1}{2}Q(u)=\frac{1}{8}B(u).
\end{eqnarray}
Thus if we assume by contradiction that $m(c)$ has a minimizer $u_c \in S(c)$ for some $c >0$  we see from Lemma \ref{lm2.3} and (\ref{3.12.2}) that
$$ 0 \geq m(c) = F(u_c) = \frac{1}{8}B(u_c) >0.$$
This contradiction ends the proof of Point (ii) and of the theorem.
\end{proof}

\begin{proof}[Proof of Theorem \ref{th1.2}] We first consider the case $p\in (3,\frac{10}{3}]$ and we assume by contradiction that there exists sequences $\{c_n\} \subset \R^+$, with $c_n \to 0$, as $n \to \infty$, and $\{u_n\} \subset S(c_n)$ such that $u_n \in S(c_n)$ is a critical point of $F(u)$ restricted to $S(c_n)$. Then since
$$Q(u_n) = A(u_n) +\frac{1}{4}B(u_n)-\frac{3(p-2)}{2p}C(u_n)=0,$$
we deduce, from Gagliardo-Nirenberg's inequality, that for some $C>0,$
\begin{equation}\label{100}
A(u_n)\leq \frac{3(p-2)}{2p}C(u_n)\leq C \cdot A(u_n)^{\frac{3(p-2)}{4}}\cdot c_n^{\frac{6-p}{4}}.
\end{equation}
Thus there holds
$$
A(u_n)^{\frac{10-3p}{4}}\leq C \cdot c_n^{\frac{6-p}{4}}
$$
and we get that
\begin{eqnarray}\label{3.13}
A(u_n)\to 0\quad \mbox{as}\quad n\to \infty
\end{eqnarray}
if $p \in (3, \frac{10}{3})$ and directly a contradiction if $p = \frac{10}{3}$. Now when $ p \in (3, \frac{10}{3})$ by Lemma \ref{lm_q} we know, since $Q(u_n)=0$, that there exists a constant $C >0$ such that
\begin{eqnarray*}
\frac{64\pi-1}{64\pi}A(u_n) \leq C \cdot A(u_n)^{\frac{3}{2}} \cdot c_n^{\frac{1}{2}}
\end{eqnarray*}
or equivalently that
\begin{equation}\label{300}
\frac{64\pi-1}{64\pi} \leq C \cdot A(u_n)^{\frac{1}{2}} \cdot c_n^{\frac{1}{2}}.
\end{equation}
But (\ref{300}) implies that
$A(u_n)\to \infty$ as $n \to \infty$ and this contradicts (\ref{3.13}).

Now when $p=3$, it is enough to prove that, for any $c>0$, there holds
\begin{eqnarray}\label{3.15}
Q(u)>0,\quad \mbox{ for all } u\in S(c).
\end{eqnarray}
Indeed, if (\ref{3.15}) holds true,  we can conclude the non-existence of minimizers directly from Lemma \ref{lm2.3}. To check (\ref{3.15}), we use (\ref{kikuchi}) with $\eta = 2$. Then, from (\ref{2.2.2}) and (\ref{kikuchi}), we get
$$\frac{1}{4}\int_{\R^3}\int_{\R^3}\frac{|u(x)|^2|u(y)|^2}{|x-y|}dx dy \geq - \frac{1}{16 \pi} ||\nabla u||_2^2 + \frac{1}{2}||u||^3_3.$$
Thus, for any $u\in S(c)$,
\begin{eqnarray*}
Q(u)&=&||\nabla u||_2^2+ \frac{1}{4}\int_{\R^3}\int_{\R^3}\frac{|u(x)|^2|u(y)|^2}{|x-y|}dx dy-\frac{1}{2}||u||^3_3\\
&\geq& ||\nabla u||_2^2 - \frac{1}{16 \pi} ||\nabla u||_2^2>0.
\end{eqnarray*}
At this point the proof is completed.
\end{proof}

\section{On the quasilinear minimization problem}

In the proofs of Theorems \ref{th4.3} and \ref{th4.33} we only provide the parts which were not established or whose proofs in \cite{CJS} contains a gap. First we observe
\begin{lem}\label{lm4.1}
Assume that $p \in [1 + \frac{4}{N}, 3+\frac{4}{N})$. If there exists a $\overline{c}>0$ such that $\overline{m}(\overline{c})=0$ is achieved, then
\begin{eqnarray}\label{4.3}
\overline{m}(c)<0,\ \ \mbox{for all }\ c>\overline{c}.
\end{eqnarray}
\end{lem}

\begin{proof} Let $\bar{u}\in \sigma(\overline{c})$ be a minimizer of $\overline{m}(\overline{c})$. Setting $(\bar{u})_t(x)=\bar{u}(t^{-\frac{1}{N}}x)$ for $t>1$,  we have $\|(\bar{u})_t\|_2^2=t\|\bar{u}\|_2^2=t\overline{c}$, and
\begin{eqnarray}\label{4.3.1}
\overline{m}(t\overline{c}) \leq \mathcal{E}((\bar{u})_t)&=& t^{1-\frac{2}{N}}\left ( \int_{\mathbb{R}^N}\frac{1}{2}|\nabla \bar{u}|^2+|\bar{u}|^2|\nabla \bar{u}|^2 dx    \right )-\frac{t}{p+1} \int_{\mathbb{R}^N}|\bar{u}|^{p+1}dx \nonumber  \\
&=& t \left [ t^{-\frac{2}{N}}\int_{\mathbb{R}^N} \left(\frac{1}{2}|\nabla \bar{u}|^2+|\bar{u}|^2|\nabla \bar{u}|^2 \right) dx -\frac{1}{p+1} \int_{\mathbb{R}^N}|\bar{u}|^{p+1}dx \right ] \\
&<& t \mathcal{E}(\bar{u})=t \overline{m}(\overline{c}). \nonumber
\end{eqnarray}
Thus \eqref{4.3} follows immediately from \eqref{4.3.1} since $\overline{m}(\overline{c})=0$.
\end{proof}



Similarly with Lemma \ref{lm3.5}, we have for $c_N$ given by (\ref{4.3'}).
\begin{lem}\label{lm4.5}
Assume that $p=3+\frac{4}{N}$. Then $c_N \in (0, \infty)$.
\end{lem}

\begin{proof} We know from (4.5) of \cite{CJS} that when $p \in [1 + \frac{4}{N}, 3 + \frac{4}{N}]$ there exists a $C>0$, depending only on $p$ and $N$, such that
\begin{eqnarray}\label{pf4.5.1}
\|u\|_{p+1}^{p+1} \leq C \cdot \|u\|_{2}^{2(1-\theta)}\cdot \left( \int_{\R^N}|u|^2|\nabla u|^2dx  \right)^{\frac{\theta N}{N-2}}, \quad \mbox{for all } \in \mathcal{X}
\end{eqnarray}
where $$\theta= \frac{(p-1)(N-2)}{2(N+2)} \quad \mbox{and} \quad  \mathcal{X} = \{ u\in H^1(\R^N):\ \int_{\R^N}|u|^2|\nabla u|^2dx< \infty \}.$$
Letting $p=3+\frac{4}{N}$ in (\ref{pf4.5.1}), we obtain that
\begin{eqnarray}\label{pf4.5.2}
\|u\|_{4+4/N}^{4+4/N} \leq C \cdot \|u\|_{2}^{\frac{4}{N}}\cdot \left( \int_{\R^N}|u|^2|\nabla u|^2dx  \right), \quad \mbox{ for all } u \in \mathcal{X}.
\end{eqnarray}
Thus, for any $u \in \sigma(c)$, there holds
\begin{eqnarray*}
\mathcal{E}(u)&\geq& \frac{1}{2}\|\nabla u\|_2^2 + \int_{\R^N}|u|^2|\nabla u|^2dx- C \cdot c^{\frac{2}{N}}\cdot \int_{\R^N}|u|^2|\nabla u|^2dx \\
&\geq& \left( 1- C \cdot c^{\frac{2}{N}} \right) \cdot \int_{\R^N}|u|^2|\nabla u|^2dx
\end{eqnarray*}
and  $\mathcal{E}(u)>0$ for all $ u\in \sigma(c)$ if $c>0$ is sufficiently small. This proves that $c_N>0$.

Now take $u_1\in \sigma(1)$ arbitrary and consider the scaling
\begin{eqnarray}\label{pf4.5.3}
u_t(x)=u_1(t^{-\frac{1}{N}x}), \quad \mbox{ for all } t>0.
\end{eqnarray}
We have $u_t \in \sigma(t)$ and
\begin{eqnarray}\label{pf4.5.4}
\mathcal{E}(u_t)&=& t^{1-\frac{2}{N}}\left( \frac{1}{2}\|\nabla u_1\|_2^2+\int_{\R^N}|u_1|^2|\nabla u_1|^2dx  \right)-t\cdot \frac{N}{4(N+1)}\|u_1\|_{4+4/N}^{4+4/N} \nonumber  \\
&=& t \left[ t^{-\frac{2}{N}}\left( \frac{1}{2}\|\nabla u_1\|_2^2+\int_{\R^N}|u_1|^2|\nabla u_1|^2dx  \right)-\frac{N}{4(N+1)}\|u_1\|_{4+4/N}^{4+4/N} \right].
\end{eqnarray}
This shows that $\mathcal{E}(u_t)<0$ for $t>0$ large and proves that $c_N< \infty$.
\end{proof}

\begin{proof}[Proof of Theorem \ref{th4.3}]
In Theorem 1.12 of \cite{CJS}, Point (i) was already proved except for the statement that $\overline{m}(c(p,N)) =0$. But it is a direct consequence of Point (ii) that we shall now prove. Let $c>0$ be arbitrary but fixed and let $\{c_n\}$ be a sequence  such that $c_n \to c$. We need to show that $\overline{m}(c_n)\to \overline{m}(c)$. By the definition of $\overline{m}(c_n)$, for each $n \in \N^+$, there exists a $u_n\in \sigma(c_n)$ such that
\begin{eqnarray}\label{4.9}
\mathcal{E}(u_n)\leq \overline{m}(c_n)+\frac{1}{n}.
\end{eqnarray}
It is shown in \cite{CJS} that  $\overline{m}(c)\leq 0$ for any $c>0$. Thus in particular
\begin{eqnarray}\label{4.11}
\mathcal{E}(u_n)\leq \frac{1}{n}.
\end{eqnarray}
Now we claim that the sequences $\{\|\nabla u_n\|_{2}^2\}$, $\{\int_{\R^N}|u_n|^2|\nabla u_n|^2dx\},$ $ \{\|u_n\|_{p+1}^{p+1}\}$ are bounded. Indeed using (\ref{4.11}) and (\ref{pf4.5.1}), we have
\begin{equation}\label{4.13}
\frac{1}{n}\geq \mathcal{E}(u_n)\geq \int_{\R^N}|u_n|^2|\nabla u_n|^2dx- \frac{C}{p+1}c_n^{1-\theta}\left( \int_{\R^N}|u_n|^2|\nabla u_n|^2dx  \right)^{\frac{\theta N}{N-2}}.
\end{equation}
Since $\frac{\theta N}{N-2}<1$ as $p \in [1 + \frac{4}{N}, 3+\frac{4}{N})$, we conclude from (\ref{4.13}) that $\{ \int_{\R^N}|u_n|^2|\nabla u_n|^2dx \}$ is bounded and then from (\ref{pf4.5.1}) that $\{ \|u_n\|_{p+1}^{p+1}\}$ is also bounded. At this point the fact that $\{ \|\nabla u_n\|_{2}^2\}$ is bounded follows from the boundedness of $\mathcal{E}(u_n)$. Now we see that
\begin{eqnarray*}\label{4.14}
\overline{m}(c)&\leq& \mathcal{E}\left( \sqrt{\frac{c}{c_n}}u_n \right ) \nonumber \\
&=&\frac{1}{2}\left( \frac{c}{c_n}\right )\|\nabla u_n\|_{2}^2+ \left( \frac{c}{c_n}\right )^2 \int_{\R^N}|u_n|^2|\nabla u_n|^2dx- \frac{1}{p+1}\left( \frac{c}{c_n}\right )^{\frac{p+1}{2}}\|u_n\|_{p+1}^{p+1}  \nonumber \\
&=&\mathcal{E}(u_n) + o(1) \leq \overline{m}(c_n)+ o(1). \nonumber
\end{eqnarray*}
On the other hand, for a minimizing sequence $\{v_m\}$ of $\overline{m}(c)$, we have
\begin{eqnarray*}
\overline{m}(c_n)\leq \mathcal{E}\left( \sqrt{\frac{c_n}{c}}v_m \right )= \mathcal{E}(v_m) + o(1) = \overline{m}(c)+ o(1).
\end{eqnarray*}
From these two estimates we deduce that $\lim_{n\to \infty}\overline{m}(c_n)=\overline{m}(c)$.  \\

We now prove Point (iii). Note that the statement in Theorem 1.12 of \cite{CJS} concerning $p = 3 + \frac{4}{N}$ was incorrect. We already know, from Lemma \ref{lm4.5}, that $c_N\in (0,\infty)$. Using the definition of $c_N$, it follows directly that $\overline{m}(c)=0$ for any $c\in (0, c_N)$, since one always has $\overline{m}(c)\leq 0$ for any $c\in (0, \infty)$. Now if $c>c_N$, we proceed as in the proof of Theorem \ref{lm1.1} (v), namely we observe that there exists a $v\in \sigma(c)$ such that $\mathcal{E}(v)\leq 0$. Indeed if we assume that
$\mathcal{E}(u)>0$ for all  $ u\in \sigma(c)$ we reach a contradiction as follows. For an arbitrary $\hat{c}\in [c_N, c)$ taking any $u\in \sigma(\hat{c})$ we scale it as in (\ref{pf4.5.3}) where $t=c/ \hat{c}$. Then $u_t\in \sigma(c)$ and it follows from  (\ref{pf4.5.4})  that
$\mathcal{E}(u_t)\leq t\mathcal{E}(u).$
This implies that $\mathcal{E}(u)>0$ for all $u\in \sigma(\hat{c})$ and since  $\hat{c}\in [c_N, c)$ is arbitrary this contradicts the definition of $c_N > 0$.

Hence, for any $c \in (c_N, \infty)$, there exists a $u_0\in \sigma(c)$ such that $\mathcal{E}(u_0)\leq 0$ and we consider the scaling
\begin{eqnarray}\label{4.15}
u^{\delta}(x)=\delta^{\frac{N}{2}}u_0(\delta x),\quad \mbox{ for all }  \delta>0.
\end{eqnarray}
Then $u^{\delta} \in \sigma(c)$, for all $\delta >0$ and
\begin{eqnarray}\label{4.16}
\mathcal{E}(u^{\delta})&=& \frac{\delta^2}{2}\|\nabla u_0\|_2^2+\delta^{N+2}\int_{\R^N}|u_0|^2|\nabla u_0|^2dx - \frac{N}{4(N+1)}\delta ^{N+2}\|u_0\|_{4+4/N}^{4+4/N} \nonumber\\
&=& \frac{\delta^2}{2}\|\nabla u_0\|_2^2 - \delta^{N+2} \left ( \frac{N}{4(N+1)}\|u_0\|_{4+4/N}^{4+4/N}- \int_{\R^N}|u_0|^2|\nabla u_0|^2dx \right ).
\end{eqnarray}
Since $\mathcal{E}(u_0)\leq 0$, necessarily
$$\frac{N}{4(N+1)}\|u_0\|_{4+4/N}^{4+4/N}- \int_{\R^N}|u_0|^2|\nabla u_0|^2dx>0$$
and thus we see from (\ref{4.16}) that $\lim_{\delta \to \infty}\mathcal{E}(u^{\delta})=-\infty$. It proves that $\overline{m}(c)=-\infty$ for any $c\in (c_N, +\infty)$.
\end{proof}

Before giving the proof of Theorem \ref{th4.33} we treat the limit case $c = c(p,N)$.

\begin{lem}\label{lm5.3}
Assume that $p \in (1+\frac{4}{N}, 3+\frac{4}{N})$. Then $\overline{m}(c(p,N))$ admits a minimizer.
\end{lem}

\begin{proof} Let $c_n:=c(p,N)+\frac{1}{n},$ for all  $n\in \N^+$. Since $\overline{m}(c_n) <0$ we know by Lemma 4.3 of \cite{CJS} that $\overline{m}(c_n)$ admits, for all $n \in \N^+$ a minimizer that is Schwartz symmetric. We claim that $\{u_n\}$ is bounded in  $\mathcal{X}$, namely that $\{u_n\}$ is bounded in $H^1(\R^N)$ and $\{\int_{\R^N}|u_n|^2|\nabla u_n|^2dx\}$ is bounded. Indeed using  \eqref{pf4.5.1} we have since $\mathcal{E}(u_n) \leq 0$, for all $n \in \N ^+$,
\begin{eqnarray}\label{5.10}
\frac{1}{2}\|\nabla u_n \|_2^2 +  \int_{\R^N}|u_n|^2|\nabla u_n|^2dx & \leq & \frac{1}{p+1}\int_{\R^N}|u_n|^{p+1}dx  \nonumber\\
&\leq & \frac{C}{p+1}c_n^{1-\theta}\cdot \left ( \int_{\R^N}|u_n|^2|\nabla u_n|^2dx \right )^{\frac{\theta N}{N-2}} .
\end{eqnarray}
Since $p \in [1+\frac{4}{N}, 3+\frac{4}{N})$ we have $\frac{\theta N}{N-2}<1$ and thus \eqref{5.10} implies that both
$\{\int_{\R^N}|u_n|^2|\nabla u_n|^2dx\}$ and $ \{||\nabla u_n||_2^2\}$
are bounded.  \\

Passing to a subsequence we can assume that $u_n\rightharpoonup u_0$ in $\mathcal{X}$. Now from Lemma 4.3 of \cite{CJS} we have that
$$T(u_0)\leq \liminf_{n\to \infty}T(u_n) \quad \mbox{where} \quad T(u):=\frac{1}{2}\left \| \nabla u \right \|_2^2+\int_{\R^N}|u|^2|\nabla u|^2dx.$$
Also the fact that $\{u_n\}$ is a sequence of Schwartz symmetric functions readily implies that $u_n \to u_0$ in $L^{p+1}(\R^N)$. Thus, since by Theorem \ref{th4.3} (ii),  $\lim_{n\to \infty}\mathcal{E}(u_n)=
\lim_{n \to \infty}\overline{m}(c_n)=0$ we obtain that $\mathcal{E}(u_0) \leq 0.$ Also since  $||u_0||_2^2 \leq c(p,N)$  necessarily $\mathcal{E}(u_0) = 0.$ \smallskip

In order to show that $||u_0||_2^2 = c(p,N)$ and thus that $u_0$ is a minimizer of $c(p,N)$ we first show that $u_0 \neq 0$. By contradiction let us assume that $u_0 =0$. Then using the fact that $u_n \to 0$ in $L^p(\R^N)$ we get from $\mathcal{E}(u_n)\to 0$ that
\begin{eqnarray}\label{5.12}
\left \| \nabla u_n \right \|_2^2\to 0 \ \ \mbox{  and  }\ \int_{\R^N}|u_n|^2|\nabla u_n|^2dx \to 0,\ \mbox{ as } n\to \infty.
\end{eqnarray}
As in the proof of Lemma \ref{lm5.1} we shall prove that $\mathcal{E}(u_n)\geq 0$ for $n\in \mathbb{N}^+$  sufficiently large and this will contradict the fact that $\mathcal{E}(u_n) = \overline{m}(c_n) < 0$ for $n\in \mathbb{N}^+$. For  $p\in (1+\frac{4}{N}, \frac{N+2}{N-2}]$ if $N\geq 3$ and $p\in (1+\frac{4}{N},+\infty)$ if $N=1,2$,  by Gagliardo-Nirenberg's inequality, we have
\begin{equation}\label{5.12.1}
\int_{\R^N}|u_n|^{p+1}dx \leq C  \|\nabla u_n\|_2^{\frac{N(p-1)}{2}}\cdot c_n^{\frac{(N+2)-(N-2)p}{4}}
\leq  C\|\nabla u_n\|_2^{\frac{N(p-1)}{2}}.
\end{equation}
Thus
\begin{eqnarray*}
\mathcal{E}(u_n)&\geq &\frac{1}{2}\|\nabla u_n\|_2^2- C \|\nabla u_n\|_2^{\frac{N(p-1)}{2}}\\
&=& \|\nabla u_n\|_2^2\left ( \frac{1}{2}-  C \|\nabla u_n\|_2^{\frac{Np-(N+4)}{2}}  \right ).
\end{eqnarray*}
This, together with \eqref{5.12}, proves that $\mathcal{E}(u_n)\geq 0$ as $n\in \mathbb{N}^+$ is sufficiently large. For  $p\in (\frac{N+2}{N-2}, 3+\frac{4}{N}), N\geq 3$, we know from the proof of Theorem 1.12 of \cite{CJS} that  $\{u_n\}$  it is bounded in $L^q(\R^N)$ for all $q\geq \frac{4N}{N-2}$. Thus by H\"older and Sobolev's inequalities we can write
\begin{eqnarray}\label{5.12.2}
\int_{\R^N}|u_n|^{p+1}dx&\leq& C(p,N) \|\nabla u_n\|_2^{\alpha}\cdot \|u_n\|_{(p-1)N}^{\beta},
\end{eqnarray}
where
$$\alpha= \frac{2N(p-1)-2(p+1)}{(p-1)(N-2)-2} \quad \mbox{and} \quad \beta=(p-1)\frac{(N-2)(p+1)-2N}{(p-1)(N-2)-2}.$$
For more details see, in particular, (4.16) in \cite{CJS}. Now since $\|u_n\|_{(p-1)N}^{\beta}$ is bounded we have
\begin{eqnarray*}
\mathcal{E}(u_n)&\geq &\frac{1}{2}\|\nabla u_n\|_2^2-C(p,N)\|\nabla u_n\|_2^{\alpha}\\
&=& \|\nabla u_n\|_2^2\left ( \frac{1}{2}-  C(p,N)\|\nabla u_n\|_2^{\alpha-2}  \right ).
\end{eqnarray*}
Since $\alpha-2>0$ as $p>1$, we then deduce using  \eqref{5.12} that $\mathcal{E}(u_n)\geq 0$ for all $n\in \mathbb{N}^+$  sufficiently large. This proves that $u_0 \neq 0$. Finally if we assume that $ \|u_0\|_2^2< c(p,N)$ we directly get a contradiction from Lemma \ref{lm4.1} since $\overline{m}(c) =0$ for all $c \in (0, c(p,N)]$. Thus $\|u_0\|_2^2=c(p,N)$ and $u_0$ is a minimizer of $\overline{m}(c(p,N))$.
\end{proof}

\begin{proof}[Proof of Theorem \ref{th4.33}]
In Theorem 1.12 of \cite{CJS} it is shown that $\overline{m}(c)$ admits a minimizer if $c \in ( c(p,N), \infty)$. By Lemma \ref{lm5.3} this is also true for $c =c(p,N)$. To complete the proof of Point (i) we need to show that for $c \in (0, c(p,N))$, $\overline{m}(c)$  does not admit a minimizer. But since $\overline{m}(c)=0$ for $c \in (0, c(p,N)]$ it results directly from Lemma
\ref{lm4.1}. To prove Point (ii) we argue by contradiction assuming that there exists a $c>0$ such that $\overline{m}(c)$ admits a minimizer $u_{c}$. Then, by standard arguments, $u_c$ satisfies weakly
\begin{equation}\label{4.4}
- \Delta u_c - \lambda_c u_c - u_c \Delta |u_c|^2 = |u_c|^{p-1}u_c,
\end{equation}
where $\lambda_c \in \R$ is the associated Lagrange multiplier. Multiplying (\ref{4.4})
by $u_{c}$ and integrating we derive that
\begin{eqnarray}\label{4.8.2}
\int_{\R^N}|\nabla u_{c}|^2dx + 4 \int_{\R^N}|u_{c}|^2|\nabla u_{c}|^2dx -
\int_{\R^N}| u_{c}|^{p+1}dx=\lambda_{c}\|u_{c}\|_{2}^2.
\end{eqnarray}
Also, from  Lemma 3.1 of \cite{CJS} we know that $u_c$ satisfies the  Pohozaev identity
\begin{equation}\label{rajout1}
\frac{N-2}{N}\left(\frac{1}{2}\int_{\R^N}|\nabla u_{c}|^2dx + \int_{\R^N}|u_{c}|^2|\nabla u_{c}|^2dx \right)=\frac{\lambda_{c}}{2}\|u_{c}\|_{2}^2+\frac{1}{p+1}\|u_{c}\|_{p+1}^{p+1}.
\end{equation}
It follows from (\ref{4.8.2}) and (\ref{rajout1}) that
\begin{eqnarray}\label{4.8.3}
\|\nabla u_{c}\|_{2}^2 + (N+2) \int_{\R^N}|u_{c}|^2|\nabla u_{c}|^2dx - \frac{N(p-1)}{2(p+1)}\|u_{c}\|_{p+1}^{p+1}=0,
\end{eqnarray}
by which we can rewrite $\mathcal{E}(u_{c})$ as
\begin{equation}\label{4.8.4}
\mathcal{E}(u_{c})= \frac{Np-(N+4)}{2N(p-1)}\|\nabla u_{c}\|_{2}^2 + \frac{Np-(3N+4)}{N(p-1)} \int_{\R^N}|u_{c}|^2|\nabla u_{c}|^2dx.
\end{equation}
When $p=3+\frac{4}{N}$, (\ref{4.8.4}) becomes
\begin{eqnarray}\label{4.8.5}
\mathcal{E}(u_{c})= \frac{N}{2N+4}\|\nabla u_{c}\|_{2}^2.
\end{eqnarray}
This is clearly a contradiction since by assumption $\mathcal{E}(u_{c})=\overline{m}(c)\leq 0$ and Point (ii) is established.
\end{proof}

\begin{proof}[Proof of Theorem \ref{th4.333}]
From the proof of Theorem \ref{th4.33}, we know that any critical point $u_c$ of $\mathcal{E}(u)$ restricted to $\sigma(c)$ must satisfy \eqref{4.8.3}. Denoting
\begin{eqnarray*}
\overline{Q}(u)=\|\nabla u\|_{2}^2 + (N+2) \int_{\R^N}|u|^2|\nabla u|^2dx - \frac{N(p-1)}{2(p+1)}\|u\|_{p+1}^{p+1},
\end{eqnarray*}
we thus have $\overline{Q}(u_c)=0$. Now we assume by contradiction that there exist sequence $\{c_n\}\subset \R^+$ with $c_n\to 0$, and $\{u_n\}\subset \sigma(c_n)$ such that $u_n$ is a critical point of $\mathcal{E}(u)$ on $\sigma(c_n)$. Then for each $n\in \N^+$, $\overline{Q}(u_n)=0$ and using \eqref{pf4.5.1} we obtain
\begin{equation}\label{4.9.1}
\|\nabla u_n\|_2^2+ (N+2) \int_{\R^N}|u_n|^2|\nabla u_n|^2dx \leq C \cdot c_n^{1-\theta} \cdot \left ( \int_{\R^N}|u_n|^2|\nabla u_n|^2dx \right )^{\frac{\theta N}{N-2}},
\end{equation}
where $\theta=\frac{(p-1)(N-2)}{2(N+2)}$.  When $p=3+\frac{4}{N}$ we have $\frac{\theta N}{N-2}=1$ , $1-\theta=\frac{4}{N}$ and thus we get immediately a contradiction from \eqref{4.9.1}. Now when $p\in [1+\frac{4}{N}, 3+\frac{4}{N})$, $\frac{\theta N}{N-2}<1$ and we derive from \eqref{4.9.1} that
\begin{equation}\label{4.9.2}
\int_{\R^N}|u_n|^2|\nabla u_n|^2dx\to 0\ \mbox{ and } \|\nabla u_n\|_2^2\to 0\    \mbox{ as } n\to \infty.
\end{equation}
Also when $p\in [1+\frac{4}{N}, \frac{N+2}{N-2}]$ if $N\geq 3$ and $p\in [1+\frac{4}{N},+\infty)$ if $N=1,2$,  we obtain from \eqref{5.12.1} that
\begin{eqnarray} \label{4.9.3}
\overline{Q}(u_n)&\geq &\|\nabla u_n\|_2^2- C \|\nabla u_n\|_2^{\frac{N(p-1)}{2}}\cdot c_n^{\frac{(N+2)-(N-2)p}{4}} \nonumber \\
&=& \|\nabla u_n\|_2^2\left ( 1-  C \|\nabla u_n\|_2^{\frac{Np-(N+4)}{2}}\cdot c_n^{\frac{(N+2)-(N-2)p}{4}} \right ).
\end{eqnarray}
Taking \eqref{4.9.2} into account \eqref{4.9.3} implies that $\overline{Q}(u_n)>0$ for $n\in \N^+$ large enough and provides a contradiction.

When $p\in (\frac{N+2}{N-2}, 3+\frac{4}{N}), N\geq 3$, using \eqref{5.12.2} and the fact that $\{\|u_n\|_{(p-1)N}^{\beta}\} $ is bounded, we have
\begin{eqnarray*}
\overline{Q}(u_n)&\geq &\|\nabla u_n\|_2^2-C(p,N)\|\nabla u_n\|_2^{\alpha}.
\end{eqnarray*}
Since $\alpha-2$ as $p>1$, using \eqref{4.9.2} we conclude that $\overline{Q}(u_n)>0$ for $n\in \N^+$ sufficiently large. Here also we have obtained a contradiction and this ends the proof.
\end{proof}



\begin{thebibliography}{99}


\bibitem{AP}{A. Azzollini, A. Pomponio, P. d'Avenia, } On the Schr\"odinger-Maxwell equations under the effect of a
general nonlinear term,  Ann. Inst. H. Poincar\'e Anal. Non Lin\'eaire 27 (2010), no. 2, 779-791.
\bibitem{AR}{A. Ambrosetti, D. Ruiz, } Multiple bound states for the Schr\"odinger-Poisson problem, Commun. Contemp. Math. 10 (2008), no. 3, 391-404.
\bibitem{BeLi}{H. Berestycki, P.L. Lions,} Nonlinear scalar field equations {I}, Arch. Ration. Mech. Anal.,  82, (1983), no. 4, 313-346.
\bibitem{BA} {C. Bardos, F. Golse, A. D. Gottlieb, N. Mauser, } Mean field dynamics of fermions and the time-dependent Hartree-Fock equation, J. Math. Pures Appl. (9) 82 (2003), no. 6, 665-683.
\bibitem{BS1}  {J. Bellazzini, G. Siciliano, } Stable standing waves for a class of nonlinear Schr\"odinger-Poisson equations, Z. Angew. Math. Phys. 62 (2011), no. 2, 267-280.
\bibitem{BS} {J. Bellazzini, G. Siciliano, } Scaling properties of functionals and existence of constrained minimizers, J. Funct. Anal. 261 (2011), no. 9, 2486-2507.
\bibitem{BJL} {J. Bellazzini, L. Jeanjean, T-J. Luo, } Existence and instability of standing waves with prescribed norm for a class of Schr\"odinger-Poisson equations, Proc. London Math. Soc., to appear. See also arXiv:1111.4668v2 [math AP] 16. May. 2012.
\bibitem{CS}{M. Caliari, M. Squassina, } On a bifurcation value related to quasilinear Schr\"odinger equations, J. Fixed Point Theory Appl., to appear. See also 	 arXiv:1111.0526v3 [math.AP] 23 Dec. 2011.
\bibitem{CJS}{M. Colin, L. Jeanjean, M. Squassina, } Stability and instability results for standing waves of quasilinear Schr\"odinger equations, Nonlinearity 23 (2010), no. 6, 1353-1385.
\bibitem{DM}{T. D'Aprile, D. Mugnai, } Solitary waves for nonlinear Klein-Gordon-Maxwell and
Schrodinger-Maxwell equations, Proc. Roy. Soc. Edinburgh Sect. A 134 (2004), no. 5, 893-906.
\bibitem{HK2} {H. Kikuchi, } Existence and stability of standing waves for Schr\"odinger-Poisson-Slater equation, Adv. Nonlinear Stud. 7 (2007), no. 3, 403-437.
\bibitem{HK1} {H. Kikuchi, } On the existence of solutions for elliptic system related to the Maxwell-Schr\"odinger equations, Nonlinear Anal. 67 (2007),  1445-1456.
\bibitem{HK} {H. Kikuchi, } Existence and orbital stability of the standing waves for nonlinear Schr\"odinger equations via the variational method, Doctoral Thesis (2008).
\bibitem{L} {P. L. Lions, } The concentration-compactness principle in the Calculus of Variation. The locally compact case, part I and II, Ann. Inst. H. Poincare Anal. Non Lineaire 1 (1984), 109-145 and 223-283.
\bibitem{L2}{P. L. Lions, }  Solutions of Hartree-Fock Equations for Coulomb Systems, Comm. Math. Phys. 109 (1987), no. 1, 33-97.
\bibitem{Lieb} {E. H. Lieb, M. Loss, } Analysis, Second edition, Graduate Studies in Mathematics, 14, American Mathematical Society, Providence, RI, 2001.
\bibitem{LS}{E. H. Lieb, B. Simon, } The Thomas-Fermi theory of atoms, molecules and solids, Advances in Math. 23 (1977), no. 1, 22-116.
\bibitem{MA} {N. J. Mauser, } The Schr\"odinger-Poisson-X$\alpha$ equation, Appl. Math. Lett. 14 (2001), no. 6, 759-763.
\bibitem{R} {D. Ruiz, } The Schr\"odinger-Poisson equation under the effect of a nonlinear local term, { J. Funct. Anal.  237  (2006),  no. 2, 655-674.}
\bibitem {R2} {D. Ruiz, } On the Schr\"odinger-Poisson-Slater System: Behavior of Minimizers, Radial and Nonradial Cases, Arch. Rational Mech. Anal. 198 (2010), no. 1, 349-368.
\bibitem{SS} {O. Sanchez, J. Soler, } Long-time dynamics of the Schr\"odinger-Poisson-Slater system, J. Statist. Phys. 114 (2004), no. 1-2, 179-204.
\bibitem{WZ} {Z. Wang, H-S. Zhou, } Positive solution for a nonlinear stationary Schr\"odinger-Poisson system in $\R^3$, Discrete Contin. Dyn. Syst. 18 (2007), no. 4, 809-816.
\bibitem{ZZ} {L. Zhao, F. Zhao, } On the existence of solutions for the Schr\"odinger-Poisson equations, J. Math. Anal. Appl. 346 (2008), no. 1, 155-169.

\end{thebibliography}
\end{document}